\documentclass{article}
\usepackage[utf8]{inputenc}
\usepackage[T1]{fontenc}
\usepackage[top=3cm, bottom=3cm, left=3cm, right=3cm, marginpar=3cm]{geometry}
\usepackage{amsfonts,amsmath,amsthm, amssymb,hyperref,enumitem,todonotes}
\newtheorem{theorem}{Theorem}
\newtheorem{claim}[theorem]{Claim}

\theoremstyle{definition}
\newtheorem{definition}[theorem]{Definition}

\let\epsilon\varepsilon

\title{Graphs without a rainbow path of length 3\thanks{This work was supported by the National Science Centre grant 2021/42/E/ST1/00193. An extended abstract announcing the results presented in this paper has been published in the Proceedings of Eurocomb’23.}}

\author{Sebastian Babiński\thanks{Faculty of Mathematics and Computer Science, Jagiellonian University, {\L}ojasie\-wicza~6, 30-348 Krak\'{o}w, Poland. E-mail: {\tt Sebastian.Babinski@alumni.uj.edu.pl}.}\and
Andrzej Grzesik\thanks{Faculty of Mathematics and Computer Science, Jagiellonian University, {\L}ojasie\-wicza~6, 30-348 Krak\'{o}w, Poland. E-mail: {\tt Andrzej.Grzesik@uj.edu.pl}.}}

\date{}

\begin{document}
\maketitle

\begin{abstract}
In 1959 Erd\H os and Gallai proved the asymptotically optimal bound for the maximum number of edges in graphs not containing a path of a fixed length. Here we study a rainbow version of their theorem, in which one considers $k \geq 1$ graphs on a common set of vertices not creating a path having edges from different graphs and asks for the maximum number of edges in each graph. We prove the asymptotically optimal bound in the case of a path on three edges and any $k \geq 1$. 
\end{abstract}

\section{Introduction}\label{sec:intro}

A classical problem in graph theory is to determine the Tur\'an number of a graph $F$, i.e., the maximum possible number of edges in graphs not containing a particular forbidden structure $F$ as a subgraph. The notable results are exact solutions for triangle by Mantel \cite{Man07} and for complete graph by Tur\'an~\cite{Tur41}, and an asymptotically optimal bound for any non-bipartite graph by Erd\H os and Stone \cite{ES46}. Not much is known for bipartite graphs, but in the case of a path it was solved asymptotically by Erd\H os and Gallai \cite{EG59} in 1951, while in 1975 Faudree and Schelp \cite{FS75} provided an exact solution.

There are many possible ways to define a rainbow version of the problem. For instance, Keevash, Mubayi, Sudakov and Verstra\"ete \cite{KMSV07}  proved that if we additionally require that the coloring is a proper edge coloring and maximize the total number of colored edges avoiding a rainbow copy of $F$, then the answer for non-bipartite graph $F$ is asymptotically the same as the Tur\'an number of $F$. Later, results for some bipartite graphs in such setting appeared in particular in \cite{DLS13, EGM19, Hal22, JPS17, JR20}, as well as results regarding maximizing  subgraphs other than edges were proven in \cite{BDHL20, GMMC22, HP21, Jan20, Jan22}.

Here we concentrate on a rainbow version without the additional assumption on proper coloring and when the number of edges in each color is maximized. Formally speaking, for a graph $F$ and an integer $k$ we consider $k$ graphs $G_1, G_2, \ldots, G_k$ on the same set of vertices and ask for the maximum possible number of edges in each graph avoiding the appearance of a copy of $F$ having every edge from a different graph. In other words, for every $i$ we color edges of $G_i$ in color $i$ (in particular it means that an edge can be in many colors) and forbid all copies of $F$ having non-repeated colors, so called rainbow copies. Note that if all $G_i$ are exactly the same, then the existence of a rainbow copy of $F$ is equivalent to the existence of a non-colored copy of~$F$, therefore any bound in the rainbow version is also bounding the Tur\'an number of $F$. 

Recently, Aharoni, DeVos, de la Maza, Montejano and \v S\'amal \cite{ADMMS20} and independently Culver, Lidick\'y, Pfender
and Volec, answering a question of Diwan and Mubayi \cite{DM07} partially solved by Magnant \cite{Mag15}, proved that for $3$ colors and $F$ being a triangle the asymptotically optimal bound is surprisingly $\left(\frac{26-2\sqrt{7}}{81}\right)n^2 \approx 0.2557n^2$.
They also asked for similar theorems for bigger cliques, other graphs and different colored patterns (in this setting some results were proven in \cite{DMM19} and \cite{LMT20}). Similar problem, but where one maximizes the sum or product of the number of edges (instead of the number of edges in each color), was considered in particular in  \cite{FGHLSTVZ22} and \cite{KSSV04}. 

Here we prove the asymptotically tight bound in the case of a path with 3 edges and any number of colors.  

\begin{theorem}\label{thm:main_asymptotic}
For every $\epsilon > 0$ there exists $n_0 \in \mathbb{N}$ such that for every $n \geq n_0$, $k \geq 1$ and graphs $G_1, G_2, \ldots, G_k$ on a common set of $n$ vertices, each graph having at least $(f(k) + \epsilon)\frac{n^2}{2}$ edges, where
\[ f(k) = \left\{
\begin{array}{ll}
\lceil \frac{k}{2} \rceil^{-2} & \text{for } k \leq 6,\\
\frac{1}{2k - 1} & \text{for } k \geq 7,
\end{array} \right.\]
there exists a rainbow path with 3 edges.
Moreover, the above bound on the number of edges is asymptotically optimal for each $k \geq 1$.
\end{theorem}

In order to avoid struggling with the lower-order error terms and to obtain a structure easier to handle, we rewrite Theorem~\ref{thm:main_asymptotic} to a bit different setting.

Assuming that Theorem~\ref{thm:main_asymptotic} does not hold we obtain an arbitrarily large counterexample with at least $(f(k) + \epsilon)\frac{n^2}{2}$ edges in each color and without a rainbow path with 3 edges. Using colored graph removal lemma implied by the Szemer\'edi Regularity Lemma (see \cite{Fox11} or Appendix C in \cite{IKLS23}), we remove all rainbow walks with 3 edges by removing at most $\frac{1}{4}\epsilon n^2$ edges in each color. Then, we add all possible edges without creating rainbow walks with 3 edges. Finally, we group all the vertices into clusters based on the colors on the incident edges. Note that if there is an edge between two clusters (or inside one), then all the vertices between these clusters (or inside this cluster) can be connected by edges in the same color without creating a rainbow walk of length 3. Thus, from the maximality, between clusters (and inside them) in each color we have none or all possible edges. Additionally, notice that vertices in a cluster incident to only one or two colors can be all connected by edges in those colors, while vertices incident to more than 2 colors need to form an independent set. 
Therefore, in order to prove Theorem~\ref{thm:main_asymptotic}, it is enough to prove its equivalent version for such kind of clustered graphs. 

\begin{definition}
For any integer $k \geq 1$ a \textit{clustered graph for $k$ colors} is an edge-colored weighted graph on $\binom{k}{2}+k+1$ vertices with vertex weights $b_{ij} = b_{ji}$ for $1\leq i < j \leq k$, $a_i$ for $i \in [k]$ and $x$, in which 
\begin{itemize}
\setlength{\parskip}{0pt}
\item $x\geq0$, $a_i\geq0$ for $i \in [k]$ and $b_{ij} \geq 0$ for every $1\leq i < j \leq k$,
\item $\displaystyle x + \sum_{1\leq i \leq k} a_i + \sum_{1\leq i < j \leq k} b_{ij} = 1$,
\item for every $i \in [k]$ the vertex of weight $a_i$ is connected in color $i$ with itself, the vertex of weight $x$ and all the vertices of weights $b_{ij}$ for $j \neq i$,
\item for every $1\leq i < j \leq k$ each vertex of weight $b_{ij}$ is connected in colors $i$ and $j$ with itself,
\item there are no other edges.
\end{itemize}
\end{definition}

We define the \textit{density of edges} in color $i \in [k]$ in a clustered graph $G$ as the sum over all pairs of vertices connected by an edge in color $i$ of the product of their weights. Denoting it by $d_i(G)$ we obtain 
\[d_i(G) = a_i^2  + 2a_ix + \sum_{j \in [k] \setminus \{i\}} b_{ij}^2 + 2\sum_{j \in [k] \setminus \{i\}} a_i b_{ij}.\]

A clustered graph for $k=3$ colors is depicted in Figure~\ref{fig:clustergraph}.
Intuitively, for every $i,j\in [k]$ the vertex of weight $b_{ij}$ represents $b_{ij}n$ vertices incident to edges colored $i$ and $j$, the vertex of weight~$a_i$ represents $a_i n$ vertices incident only to edges colored $i$, and the vertex of weight $x$ represents the remaining vertices forming an independent set. 

\begin{figure}[ht]
    \centering
    \includegraphics[scale=0.88]{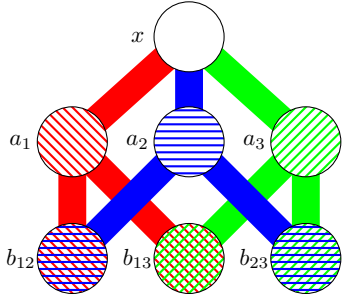}
    \caption{A clustered graph for 3 colors.}
    \label{fig:clustergraph}
\end{figure}

The equivalent version of Theorem~\ref{thm:main_asymptotic} for clustered graphs is the following.

\begin{theorem}\label{thm:main_weight}
For every integer $k \geq 1$, if $G$ is a clustered graph for $k$ colors, then
\[ \min_{i \in [k]} d_i(G) \leq f(k),
\text{ where }
 f(k) = \left\{
\begin{array}{ll}
\lceil \frac{k}{2} \rceil^{-2} & \text{for } k \leq 6,\\
\frac{1}{2k - 1} & \text{for } k \geq 7.
\end{array} \right.\]
\end{theorem}

Theorem~\ref{thm:main_asymptotic} follows from Theorem~\ref{thm:main_weight}, because, as discussed after the statement of Theorem~\ref{thm:main_asymptotic}, a possible counterexample leads to a graph with density of edges in each color at least $(f(k) + \frac{1}{2}\epsilon)\frac{n^2}{2}$ and clusters of vertices behaving as weighted vertices of a related clustered graph. Dividing each cluster size by $n$ we obtain a clustered graph with density of edges in each color at least $f(k) + \frac{1}{2}\epsilon$, which contradicts Theorem~\ref{thm:main_weight}. 
Note that also Theorem~\ref{thm:main_asymptotic} implies Theorem~\ref{thm:main_weight} as any clustered graph~$G$ contradicting Theorem~\ref{thm:main_weight} having $d_i(G) \geq f(k) + 2\epsilon$ for each $i\in[k]$ and some $\epsilon>0$ leads for any appropriately large $n$ to a graph on $n$ vertices with at least $(f(k) + \epsilon)\frac{n^2}{2}$ edges in each color and no rainbow path with 3 edges (the removed $\epsilon\frac{n^2}{2}$ term was used to compensate lower order terms), which contradicts Theorem~\ref{thm:main_asymptotic}.

The bound provided in Theorem~\ref{thm:main_weight} is tight for every integer $k \geq 1$, because it is possible to construct a clustered graph for $k$ colors $G$ such that $\min_{i \in [k]} d_i(G) = f(k)$:
\begin{itemize}[nosep, label=---]
\setlength{\parskip}{1pt} 
    \item for $k=1$ let $a_1 = 1$;
    \item for $k=2$ let $b_{12} = 1$;    
    \item for $k=3$ let $b_{12} = b_{13} = \frac{1}{2}$;    
    \item for $k=4$ let $b_{12} = b_{34} = \frac{1}{2}$;  
    \item for $k=5$ let $b_{12} = b_{34} = b_{15} = \frac{1}{3}$;  
    \item for $k=6$ let $b_{12} = b_{34} = b_{56} = \frac{1}{3}$;  
    \item for $k = 5$ or $k \geq 7$ let $a_i = \frac{1}{2k - 1}$ for each $i \in [k]$, $x = \frac{k-1}{2k-1}$.
\end{itemize}
In each case the remaining weights are equal to $0$.  

For $k = 5$ there are two different types of constructions because in this case $\lceil \frac{k}{2} \rceil^{-2} = \frac{1}{2k - 1} = \frac{1}{9}$. They are depicted in Figure~\ref{fig:extremal5}. 
Note also that for $k=3$ and $k=5$ these are not the only possible constructions, as instead of having a vertex of weight $b_{1k} = \lceil \frac{k}{2} \rceil^{-1}$, one can also have for any $i\in[k-1]$ two vertices of arbitrary weights $a_k$ and $b_{ik}$ summing up to $\lceil \frac{k}{2} \rceil^{-1}$.

\begin{figure}[ht]
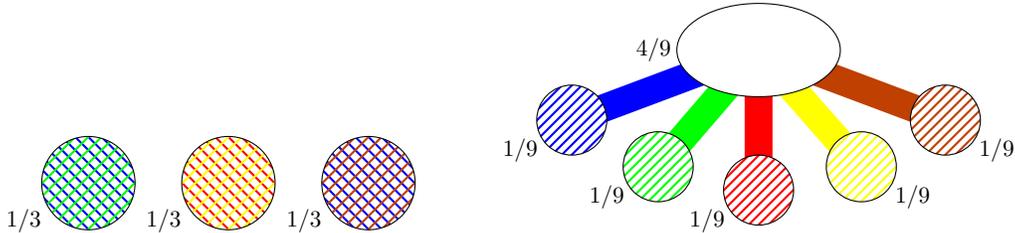

    \centering
    \includegraphics[scale=0.88]{figure-1.mps}\hspace{10mm}
    \includegraphics[scale=0.88]{figure-2.mps}
    \caption{Two possible types of extremal constructions for $k=5$.}
    \label{fig:extremal5}
\end{figure}

Theorem~\ref{thm:main_weight} for $k \in \{1, 2\}$ is trivial as then $f(k) = 1$. 
In order to prove Theorem \ref{thm:main_weight} for $k \geq 3$, due to different extremal constructions, we consider three cases depending on the value of $k$. In Section~\ref{sec:general} we provide a series of claims useful in many of the considered cases. In Section~\ref{sec:34} we prove the case $k \in \{3, 4\}$, in Section~\ref{sec:56} we consider $k \in \{5, 6\}$, while in Section~\ref{sec:7+} we deal with the remaining case~$k \geq 7$.

\section{General claims}\label{sec:general}

We will prove Theorem~\ref{thm:main_weight} by induction. As mentioned in the previous section, for $k \in \{1, 2\}$ the theorem holds. Let us fix $k \geq 3$ as the smallest integer for which the theorem does not hold. Take a clustered graph for $k$ colors $G$ maximizing the value of $\min_{i \in [k]} d_i(G)$ and, among such, maximizing the density of edges in any color.
From the maximality of $G$, there exists no clustered graph for $k$ colors $G'$ having $\min_{i \in [k]} d_i(G) < \min_{i \in [k]} d_i(G')$, or having bigger density of edges in one of the colors and non-smaller densities of edges in all the other colors. 

For shortening we denote $d_i = d_i(G)$, $b_i = \sum_{j \in [k], j \neq i} b_{ij}$, and $c_i = a_i + b_i + x$. In other words, $b_i$ is the total weight of vertices incident to color $i$ and some one another color, while $c_i$ is the total weight of vertices incident to color $i$ (including $x$ even if $a_i=0$). Additionally, let $c = \min_{i \in [k]} c_i$.

Assumption that Theorem~\ref{thm:main_weight} does not hold implies that $d_i > f(k)$ for every $i \in [k]$.
We will prove a series of claims on weights of the vertices of $G$, which will be used in further sections to obtain a contradiction for each value of $k$. 

\begin{claim}\label{cla:gen:basic_bound}
For every $i \in [k]$ it holds
\[c_i > \sqrt{f(k) + 2 b_i x + x^2} \geq \sqrt{f(k) + x^2}.\]
\end{claim}

\begin{proof}
For every $i \in [k]$ we have
\[ f(k) < d_i = a_i^2 + 2a_i b_i + \sum_{j \in [k] \setminus \{i\}} b_{ij}^2 + 2a_i x \leq (a_i + b_i + x)^2 - 2 b_i x - x^2 = c_i^2 - 2b_i x - x^2.\]
Thus, $c_i > \sqrt{f(k) + 2b_i x + x^2}$ for every $i \in [k]$ as desired.
\end{proof}

\begin{claim}\label{cla:gen:aiajbij}
For every $i, j \in [k]$, $i \neq j$ it holds 
\[a_i + a_j + b_{ij} < 1 - \sqrt{\frac{f(k)}{f(k-2)}}.\]
\end{claim}

\begin{proof}
Without loss of generality let $i = k-1$ and $j = k$. 
We will construct a clustered graph for $k-2$ colors $G'$ using the clustered graph $G$ intuitively by removing vertices of weights $a_{k-1}$, $a_k$ and $b_{(k-1)k}$, removing edges in colors $k-1$ and $k$, and rescaling all the weights to be summing to~$1$. Formally, we define the weights of the clustered graph~$G'$ as follows
\begin{eqnarray*}
b'_{ij} &=& \frac{b_{ij}}{1 - (a_{k-1} + a_k + b_{(k-1) k})} \text{ for } i, j \in [k-2], i \neq j,\\
a'_i &=& \frac{a_i + b_{i(k-1)} + b_{ik}}{1 - (a_{k-1} + a_k + b_{(k-1) k})} \text{ for } i \in [k - 2],\\
x' &=& \frac{x}{1 - (a_{k-1} + a_k + b_{(k-1) k})}.
\end{eqnarray*}

Let $p \in [k-2]$ be such that $d_p(G') = \min_{i \in [k-2]} d_i(G')$. 
Since the weight of each vertex in $G'$ is at least $\frac{1}{1-(a_{k-1} + a_k + b_{(k-1) k})}$ times bigger then the weight of the respective vertex in $G$, we have 
\[d_p(G') \geq \left(\frac{1}{1-(a_{k-1} + a_k + b_{(k-1) k})}\right)^2 d_p.\] 
Together with the inductive assumption we obtain
\[f(k-2) \geq d_p(G') \geq \left(\frac{1}{1 - (a_{k-1} + a_k + b_{(k-1) k})}\right)^2 d_p > \left(\frac{1}{1 - (a_{k-1} + a_k + b_{(k-1) k})}\right)^2f(k).\]
Rearranging the above inequality we get
\[ a_{k-1} + a_k + b_{(k-1)k} < 1 - \sqrt{\frac{f(k)}{f(k-2)}}.\qedhere\]
\end{proof}

\begin{claim}\label{cla:gen:two_ci's}
For every $i, j \in [k]$, $i \neq j$ it holds 
\[ \max\{c_i, c_j\} > \sqrt{f(k) - \left( 1 - \sqrt{\frac{f(k)}{f(k-2)}} \right) x} + x.\]
\end{claim}

\begin{proof}
Bounding the density of edges in colors $i$ and $j$ we obtain
\[ f(k) < d_\ell \leq (a_\ell + b_\ell)^2 + 2a_\ell x, \text{ for } \ell \in \{i, j\}.\]
Summing up these inequalities and using the estimate $a_i + a_j < 1 - \sqrt{\frac{f(k)}{f(k-2)}}$ from Claim \ref{cla:gen:aiajbij} we obtain
\[2f(k) < (a_i + b_i)^2 + (a_j + b_j)^2 + 2\left(1 - \sqrt{\frac{f(k)}{f(k-2)}}\right) x,\]
which implies
\[\max\{a_i + b_i, a_j + b_j\} > \sqrt{f(k) - \left(1 - \sqrt{\frac{f(k)}{f(k-2)}}\right)x}.\]
As $c_\ell = a_\ell + b_\ell + x$ for $\ell \in \{i, j\}$ from the above estimate we get
\[\max\{c_i, c_j\} > \sqrt{f(k)-\left(1-\sqrt{\frac{f(k)}{f(k-2)}}\right) x} + x,\]
as desired.
\end{proof}

\begin{claim}\label{cla:gen:numbij}
If $a_ix = 0$ for some $i \in [k]$, then there exist $j, \ell \in [k] \setminus \{i\}$, $j \neq \ell$ such that $b_{ij} > 0$ and $b_{i\ell} > 0$.
\end{claim}

\begin{proof}
Assuming that the claim does not hold, we get that $a_ix = 0$ and for some $j \neq i$ and every $\ell \neq i, j$ we have $b_{i\ell}=0$. This, together with Claim~\ref{cla:gen:aiajbij}, implies that 
\[f(k) < d_i = (a_i + b_{ij})^2 < \left( 1 - \sqrt{\frac{f(k)}{f(k-2)}} \right)^2.\]
It means that
\[\sqrt{\frac{f(k)}{f(k-2)}} + \sqrt{f(k)} < 1.\]

It remains to notice that for $k \in \{3, 4, 5, 6\}$ the left-hand side of the above inequality is equal to $1$, for $k \in \{7, 8\}$ it is equal to $\frac{4}{\sqrt{13}}$ and $\frac{4}{\sqrt{15}}$, respectively, which are both greater than $1$, while for $k \geq 9$ we obtain
\[\sqrt{\frac{2k - 5}{2k -1}} + \sqrt{\frac{1}{2k - 1}}  = \frac{\sqrt{2k-5} + 1}{\sqrt{2k-1}} > 1.\]
Therefore, we have a contradiction for each $k \geq 3$.
\end{proof}

\begin{claim}\label{cla:gen:non_zero_bij}
There exist $i, j \in [k]$, $i \neq j$ such that $b_{ij} > 0$.
\end{claim}

\begin{proof}
Assuming the contrary, that $b_{ij} = 0$ for every $i, j \in [k], i \neq j$, we obtain that the density of edges in any color $i\in[k]$ is equal to $a_i^2 + 2a_ix$. Let $p \in [k]$ be the color of the minimum density, then $a_p$ is the smallest weight out of $a_i$, for $i \in[k]$. As $a_p$ is at most $\frac{1-x}{k}$, it implies that 
\[f(k) < d_i \leq \left(\frac{1-x}{k}\right)^2 + 2\left(\frac{1-x}{k}\right)x.\]
This quadratic expression is maximized for $x = \frac{k-1}{2k-1}$, which gives $f(k) < \frac{1}{2k-1}$ and a contradiction.
\end{proof}

Claim~\ref{cla:gen:non_zero_bij} allows to define 
\[b = \min\{b_{ij}: i, j \in [k], i \neq j, b_{ij} > 0\}.\] 
The knowledge on the number of non-zero values $b_{ij}$ can be used to prove other useful bounds.

\begin{claim}\label{cla:gen:aibi_when_x_zero}
If $a_ix = 0$ for some $i \in [k]$, then 
\[a_i + b_i > \sqrt{f(k) + 2\!\!\sum_{j, \ell \in [k] \setminus \{i\}, j \neq \ell}\!\! b_{ij} b_{i\ell}} \geq \sqrt{f(k) + 2b^2}.\]
\end{claim}

\begin{proof}
Since $a_ix = 0$, it is possible to express $d_i$ in the following way
\[f(k) < d_i = a_i^2 + \sum_{j \in [k] \setminus \{i\}} b_{ij}^2 + 2\sum_{j \in [k] \setminus \{i\}} a_i b_{ij} =(a_i + b_i)^2 - 2\!\!\sum_{j,\ell \in [k] \setminus \{i\}, j \neq \ell}\!\! b_{ij} b_{i\ell}.\]
It implies that
\[a_i + b_i > \sqrt{f(k) + 2\!\!\sum_{j, \ell \in [k] \setminus \{i\}, j \neq \ell}\!\! b_{ij} b_{i\ell}}.\]
From Claim \ref{cla:gen:numbij} we know that there exist $j, \ell \in [k] \setminus \{i\}$, $j \neq \ell$ such that $b_{ij} > 0$, $b_{i\ell} > 0$ and both of them are at least $b$, so
\[\sqrt{f(k) + 2\!\!\sum_{j, \ell \in [k] \setminus \{i\}, j \neq \ell} b_{ij} b_{i\ell}} \geq \sqrt{f(k) + 2b^2},\]
as desired.
\end{proof}

In order to provide more bounds, we need to introduce the operation of removing and adding weights in a clustered graph for $k$ colors. 
Intuitively, we remove a tiny weight from some of the vertices of positive weight and add it to different vertices. From the maximality of $G$, such operation cannot enlarge the density of edges in each color, so the density of edges in at least one color needs to drop down (or the densities of edges in every color remain the same). Moreover, since the performed change is arbitrarily tiny, we do not need to calculate the exact change of the density of edges in each color, but only its main term of behavior.

\begin{definition}
For a subset $S \subset V(G)$ of vertices of positive weights, we say that we \emph{remove} weights~$w_v$ from $v$ for $v \in S$ and \emph{add} weights $w'_u$ to $u$ for $u \in T \subset V(G)$, where $\sum_{v \in S} w_v = \sum_{u \in T} w'_u$, if we construct a new clustered graph $H_\epsilon$ from $G$ by subtracting the weight $\epsilon w_v$ from the weight of~$v$ for $v \in S$, and adding the weight $\epsilon w'_u$ to the weight of $u$ for $u \in T$, where $\epsilon$ is an arbitrary small positive number. 
The difference of the densities of edges in each color in $H_\epsilon$ and $G$ is a polynomial function of~$\epsilon$. 
By the \emph{increment in color} $i \in [k]$ we define a half of the coefficient of the linear term in this difference. Since $\epsilon$ can be arbitrarily small and $G$ is maximal, it is impossible that the increment is positive in each color appearing on edges incident to $S$ and $T$.
\end{definition}

To illustrate how one can use the above operation, we prove the following claim.

\begin{claim}\label{cla:gen:aiaj2bi_lower_bound}
If $b_{ij} > 0$ for some $i, j \in [k]$, $i \neq j$, then $a_i + a_j + 2b_{ij} \geq \min\{c_i, c_j\}$.
\end{claim}

\begin{proof}
Consider removing weight $a_i + a_j + 2b_{ij}$ from the vertex of weight $b_{ij} > 0$ and adding weight $a_i + b_{ij}$ to the vertex of weight $a_i$ and weight $a_j + b_{ij}$ to the vertex of weight $a_j$. 
The increment in color $i$ is equal to 
\[-(a_i+b_{ij})(a_i + a_j + 2b_{ij}) + (a_i + b_i + x)(a_i+b_{ij}) = (a_i+b_{ij})(c_i - (a_i+a_j+2b_{ij})).\]
Similarly, the increment in color $j$ is equal to $(a_j+b_{ij})(c_j - (a_i+a_j+2b_{ij}))$.

If $a_i + a_j + 2b_{ij} < \min\{c_i, c_j\}$, then both those increments are positive, which contradicts the maximality of $G$. Therefore, $a_i + a_j + 2b_{ij} \geq \min \{c_i, c_j\}$.
\end{proof}

Using similar approach we can prove useful lower bounds for the sum of the weights~$a_i$.

\begin{claim}\label{cla:gen:sum_of_ai}
If $x > 0$, then $\sum_{i \in [k]} a_i \geq c$. 
If $x = 0$, then $\sum_{i \in [k]} a_i \geq c - 2b$.
\end{claim}

\begin{proof}
Firstly, let us assume that $x > 0$. Observe that it is not possible that $a_i = 0$ for every $i \in [k]$, since otherwise, removing a unit weight from the vertex of weight $x$ and adding weight $\frac{1}{k}$ to each vertex of weight $a_i$ for $i \in [k]$ gives positive increments in every color contradicting the maximality of~$G$.
Thus, consider removing weight $\sum_{i \in [k]} a_i$ from the vertex of weight $x$ and adding weight $a_i$ to each vertex of non-zero weight $a_i$ for $i \in [k]$. 
The density of edges in color $i \in [k]$ for which $a_i = 0$ has not changed. Therefore, there exists $j \in [k]$ such that $a_j > 0$ and the increment in color $j$ is non-positive, i.e., 
\[-a_j\sum_{i \in [k]} a_i + (a_j + b_j + x)a_j \leq 0.\] 
It implies $\sum_{i \in [k]} a_i \geq a_j+b_j+x \geq c$ as desired.

Now we assume that $x = 0$. From Claim \ref{cla:gen:non_zero_bij} we know that there exist $p, q \in [k]$ such that $p \neq q$ and $b_{pq} = b > 0$.
Then,
\[\sum_{i \in [k]} a_i \geq a_p + a_q = a_p + a_q + 2b_{pq} - 2b_{pq} \geq c - 2b_{pq} = c - 2b,\]
where the last inequality comes from Claim \ref{cla:gen:aiaj2bi_lower_bound}.
\end{proof}

The last general claim gives an upper bound on $x$.

\begin{claim}\label{cla:gen_bound_for_x}
$x < 1 - \sqrt{\frac{f(k)}{f(k-2)}}$.
\end{claim}

\begin{proof}
From Claim \ref{cla:gen:non_zero_bij} there exist $i, j \in [k], i \neq j$, for which $b_{ij} > 0$. 
Remove weight $2$ from the vertex of weight $b_{ij}$ and add unit weight to each vertex of weight $a_i$ and $a_j$. The increment in color $i$ is equal to 
\[x + a_i + b_i - 2(a_i + b_{ij}) \geq x - a_i - b_{ij} \geq x - (a_i + a_j + b_{ij}).\]
Similarly, the increment in color $j$ is at least $x - (a_i + a_j + b_{ij})$. 
Therefore, the above value must be non-positive, so $x \leq a_i + a_j + b_{ij}$. 
Additionally $a_i + a_j + b_{ij} < 1 - \sqrt{\frac{f(k)}{f(k-2)}}$ from Claim \ref{cla:gen:aiajbij}, which gives the desired bound on $x$. 
\end{proof}

\section{Three and four colors}\label{sec:34}

Firstly we are going to finish the proof of Theorem~\ref{thm:main_weight} if $k=3$. In this case $f(3) = \frac{1}{4}$, so our conjectured clustered graph $G$ satisfies $\min_{i \in [3]} d_i > \frac{1}{4}$, which implies $c_i > \frac{1}{2}$ for every $i \in [3]$. 

\begin{claim}\label{cla:3col:exists_zero_ai}
There exists $i \in [3]$ such that $a_i = 0$.
\end{claim}

\begin{proof}
Let us assume by contradiction that $a_i > 0$ for every $i \in [3]$. We remove a unit weight from each vertex of weight $a_i$ for $i \in [3]$ and add the removed weights to every vertex of the clustered graph (including the vertices of weight $a_i$) proportionally to its weight, i.e., weight $3a_i$ to each vertex of weight $a_i$, weight $3b_{ij}$ to each vertex of weight $b_{ij}$ and weight $3x$ to the vertex of weight $x$. 
The increment in color $1$ is equal to
\[3a_1(a_1 + b_{12} + b_{13} + x) + 3b_{12}(a_1 + b_{12}) + 3b_{13}(a_1 + b_{13}) + 3xa_1 - c_1 = 3d_1 - c_1.\]
Similarly, the increments in colors $2$ and $3$ are equal to $3d_2 - c_2$ and $3d_3 - c_3$ respectively.

To avoid contradiction, there exists $i \in [3]$ such that $3d_i - c_i \leq 0$, which implies that $c_i > \frac{3}{4}$ as $d_i >\frac{1}{4}$. Without loss of generality let us assume that $c_3 > \frac{3}{4}$.

If $x > 0$, then from Claim \ref{cla:gen:basic_bound}, $c_i > \sqrt{\frac{1}{4} + x^2}$ for $i \in [2]$ and from Claim \ref{cla:gen:sum_of_ai}, $\sum_{i \in [3]} a_i \geq c > \sqrt{\frac{1}{4}+x^2}$. 
Summing up the obtained inequalities for $c_i$, $i \in [3]$, and for $\sum_{i = 1}^3 a_i$, we obtain on the left-hand side each of the terms $a_i$ and $b_{ij}$ twice, while $x$ three times. Hence, we get
\[2 + x > \frac{3}{4} + 3\sqrt{\frac{1}{4}+x^2}.\]
This inequality has no solutions, which gives a contradiction.

If $x = 0$, then $b_{12} > 0$ as otherwise $a_1 + b_{13} = c_1 > \frac{1}{2}$ and $a_2 + b_{23} = c_2 > \frac{1}{2}$, which is not possible. 
Consider now removing a unit weight from the vertex of weight $b_{12}$ and adding it to each vertex of weight $a_i, b_{ij}$ for $i, j \in [3], i \neq j$ proportionally to its weight. We get a positive increment in color~$3$, while the increments in colors~$1$~and~$2$ are equal to $d_1 - (a_1 + b_{12})$ and $d_2 - (a_2 + b_{12})$, respectively. 
We get a contradiction if both of them are positive, so from symmetry we can assume that 
$d_1 - (a_1 + b_{12}) \leq 0$. It implies that $a_1 + b_{12} > \frac{1}{4}$, which gives a contradiction with $c_3 > \frac{3}{4}$.
\end{proof}

Without loss of generality we can us assume that $a_3 = 0$. 
If $x > 0$, then we can remove a unit weight from the vertex of weight $x$ and add it to the vertex of weight $b_{12}$. This is not changing the density of edges in color $3$, while the increments in colors $1$ and $2$ are positive, which is a contradiction. Therefore, $x$ must be equal to $0$ as well.

Knowing that $a_3 = x = 0$ we can show that all other weights must be positive.

\begin{claim}\label{cla:3col:allnonzero}
Values $b_{12}$, $b_{13}$, $b_{23}$, $a_1$ and $a_2$ are positive.
\end{claim}

\begin{proof}
If $b_{12} = 0$, then either $a_1+b_{13} \leq \frac{1}{2}$ or $a_2+b_{23} \leq \frac{1}{2}$. This gives that the density of edges in color $1$ or~$2$ is at most $\frac{1}{4}$, which is a contradiction. 

If $b_{13} = 0$, then either $a_1+b_{12} \leq \frac{1}{2}$ or $b_{23} \leq \frac{1}{2}$. This implies that the density of edges in color $1$ or~$3$ is at most $\frac{1}{4}$, which is a contradiction. The case $b_{23} = 0$ is analogous.

From symmetry, it remains to consider the case $a_2 = 0$. Without loss of generality we can assume that $b_{12} \leq b_{13}$. Since $\frac{1}{4} < d_2~=~b_{12}^2 + b_{23}^2$ and $b_{23}+2b_{12} \leq 1$, we obtain that $4b_{23}^2 + (1-b_{23})^2 > 1$, which implies that $b_{23} > \frac{2}{5}$.
Using the density of edges in color $1$ we have 
\[\frac{1}{4} < (a_1+b_{12}+b_{13})^2-2b_{12}b_{13} \leq (1-b_{23})^2 - 2b_{12}^2.\] 
Together with the previous bound $b_{12}^2 + b_{23}^2 > \frac{1}{4}$ coming from the density of edges in color $2$, we obtain $(1-b_{23})^2 + 2b_{23}^2 > \frac{3}{4}$, which implies $b_{23} < \frac{1}{6}$ or $b_{23} > \frac{1}{2}$. In the first case we have a contradiction with the previously proven bound $b_{23} > \frac{2}{5}$, while in the second case we have $a_1+b_{12}+b_{13} < \frac{1}{2}$, which means that the density of edges in color $1$ is smaller than $\frac{1}{4}$.
\end{proof}

If the density of edges in color $1$ (or $2$) is strictly larger than the density of edges in color $3$, then we can remove some weight from the vertex of weight $a_1$ (it is positive from Claim~\ref{cla:3col:allnonzero}) and add it to the vertex of weight $b_{23}$ (or from $a_2$ to $b_{13}$). In this way we obtain a clustered graph for $3$ colors with a larger density of edges in the least color, which contradicts the choice of $G$. Thus, we may assume that the density of edges in color $3$ is at least as big as the density of edges in color $1$ and $2$. 
This means that $b_{13}^2 + b_{23}^2 \geq b_{13}^2 + (a_1+b_{12})^2 + 2a_1b_{13}$, which implies $b_{23} > a_1 + b_{12}$ using Claim~\ref{cla:3col:allnonzero}. Therefore, $c_2 = a_2 + b_{12} + b_{23} > a_1 + a_2 + 2b_{12}$.
Similarly, $c_1 > a_1 + a_2 + 2b_{12}$. This contradicts Claim~\ref{cla:gen:aiaj2bi_lower_bound} and ends the proof for $k = 3$.

The proof of Theorem \ref{thm:main_weight} for $k = 4$ is a simple corollary of the theorem for $k = 3$ since $f(4) = f(3)$. 
Let us remove the vertex of weight $a_4$, edges in color $4$, and scale the weights of the remaining vertices so that the weights sum up to $1$ and the proportions between them are kept. Then we obtain a clustered graph for 3 colors with at least the same density of edges in each color, so a counterexample for 4 colors implies a counterexample for 3 colors, which does not exist. 

After finishing the proof we learned that Frankl, Gy\H ori, He, Lv, Salia, Tompkins, Varga and Zhu~\cite{FGHLSTVZ22} solved the problem of maximum possible product of the numbers of edges without a  rainbow path with 3 edges in the case of 3 and 4 colors. The results for 3 colors are independent (none of them implies the other), but since for 4 colors the optimal construction has the same number of edges in each color, their result implies our result for 4 colors. 

\section{Five and six colors}\label{sec:56}

In the previous section we proved Theorem~\ref{thm:main_weight} for at most $4$ colors, so now let us assume that $k = 5$. In this case $f(5) = \frac{1}{9}$ and for every $i \in [5]$ we have $d_i > \frac{1}{9}$ and 
\begin{equation}\label{eq:5col_ci}
    c_i > \sqrt{\frac{1}{9} + 2b_ix + x^2} \geq \sqrt{\frac{1}{9} + x^2}.
\end{equation}
from Claim~\ref{cla:gen:basic_bound}.
We start with two claims lower bounding $b_i$ and $c_i$.

\begin{claim}\label{cla:5col_bij}
If $b_{ij} > 0$ for some $i, j \in [k]$, $i \neq j$, then 
\[b_{ij} > \frac{1}{3}\left(\sqrt{1 - 6 x + 18 x^2} - 1 + 3 x\right) \quad \text{and} \quad c_i > \frac{1}{3}\sqrt{1-6x+27x^2 + 6x\sqrt{1 - 6 x + 18 x^2}}.\]
\end{claim}

\begin{proof}
Consider removing weight $a_i + a_j + 2b_{ij}$ from the vertex of weight $b_{ij}$ and adding weight $a_i + b_{ij}$ to the vertex of weight $a_i$ and weight $a_j + b_{ij}$ to the vertex of weight $a_j$. 
The increment in color $i$ is equal to 
\[-(a_i+b_{ij})(a_i + a_j + 2b_{ij}) + (a_i + b_i + x)(a_i+b_{ij}) = (a_i+b_{ij})(c_i - b_{ij} - (a_i+a_j+b_{ij})).\]
Using \eqref{eq:5col_ci} and $a_i + a_j + b_{ij} < \frac{1}{3}$ from Claim~\ref{cla:gen:aiajbij} we have that the increment in color $i$ is bigger than
\[(\min\{a_i, a_j\} +b_{ij})\left(\sqrt{\frac{1}{9} + 2b_{ij}x + x^2} - b_{ij} - \frac{1}{3}\right)\]
and the same value is bounding the increment in color $j$.

If $\sqrt{\frac{1}{9} + 2b_{ij}x + x^2} - b_{ij} - \frac{1}{3}$ is non-negative, then the considered operation is enlarging the density of edges in color $i$ and in $j$, while not changing the densities of edges in the remaining colors. Hence, we have
\[\sqrt{\frac{1}{9} + 2b_{ij}x + x^2} - b_{ij} - \frac{1}{3} < 0.\] 

Solving this inequality we obtain the wanted lower bound for $b_{ij}$ and as a consequence of \eqref{eq:5col_ci} also the wanted lower bound for $c_i$.
\end{proof}

\begin{claim}\label{cla:5col_bi_when_ai_empty}
If $a_i = 0$ for some $i\in[5]$, then
\[b_i > \sqrt{\frac{1}{9} + 2b\left(\frac{1}{3}-b\right)} \quad \text{and} \quad b_i > \frac{1}{3}\sqrt{-5+30x-54x^2 + (6 - 12x)\sqrt{1-6x+18x^2}}.\]
\end{claim}

\begin{proof}
Let us define $b_{ij_0}$ as $\min \{b_{ij}: j \in [5] \setminus \{i\}, b_{ij} > 0\}$. From Claim \ref{cla:gen:aibi_when_x_zero} we obtain
\[b_i > \sqrt{\frac{1}{9} + 2\!\!\sum_{j, \ell \in [5] \setminus \{i\}, j \neq \ell}\!\! b_{ij} b_{i\ell}} \geq \sqrt{\frac{1}{9} + 2\sum_{\ell \in [5] \setminus \{i, j_0\}} b_{ij_0}b_{i\ell}} = \sqrt{\frac{1}{9} + 2 b_{ij_0}(b_i - b_{ij_0})}.\]
From Claim~\ref{cla:gen:numbij} we have $b_{ij_0} \leq \frac{1}{2}b_i$, which means that the function $(0; \frac{1}{2}b_i] \ni b_{ij_0} \longmapsto 2b_{ij_0}(b_i - b_{ij_0}) \in \mathbb{R}$ is increasing. Using $b_{ij_0} \geq b$ and $b_i > \frac{1}{3}$, we obtain the first bound, while using Claim~\ref{cla:5col_bij} and $b_i > \frac{1}{3}$, we obtain the second bound.
\end{proof}

Now we can show that $x$ must be positive. 

\begin{claim}\label{cla:5col:xpositive}
$x>0$.
\end{claim}

\begin{proof}
Assume that $x=0$. From Claim \ref{cla:gen:numbij} for every $i \in [5]$ the set $ \{j \in [5] \setminus\{i\}: b_{ij} > 0\}$ has at least two elements. 

If at most one of the weights $a_i$ is zero (without loss of generality $a_i > 0$ for $i \in [4]$), we remove a unit weight from each vertex of weight $a_i$ for $i \in [4]$ and add weights to each vertex proportionally to its weight (i.e., for each vertex of weight $a_i$ or $b_{ij}$ we add $4a_i$ or $4b_{ij}$ respectively). The increment in color $i \in [4]$ is at least $4d_i - c_i$, while the density of edges in color $5$ increases. From the maximality of~$G$ there exists $\ell \in [4]$ such that $4d_\ell - c_\ell \leq 0$. Without loss of generality let $\ell = 1$, and so $c_1 \geq 4d_1 > \frac{4}{9}$. Additionally we know that $c_i > \sqrt{\frac{1}{9} + 2b^2}$ for $i \in [5]\setminus\{1\}$ and $\sum_{i \in [5]} a_i > \frac{1}{3} - 2b$ from Claims \ref{cla:gen:aibi_when_x_zero} and~\ref{cla:gen:sum_of_ai}. By summing up all these inequalities, we obtain
\[2 > \frac{4}{9} + 4\sqrt{\frac{1}{9} + 2b^2} + \frac{1}{3} - 2b,\]
which is a contradiction.

Now we know that the set $\{i \in [5]: a_i = 0\}$ has at least two elements. Without loss of generality $a_1 = a_2 = 0$. From Claim \ref{cla:5col_bi_when_ai_empty}, $b_i > \sqrt{\frac{1}{9} + 2b(\frac{1}{3}-b)}$ for $i \in \{1,2\}$. Using $c_i > \sqrt{\frac{1}{9} + 2b^2}$ for $i \in \{3, 4, 5\}$ from Claim \ref{cla:gen:aibi_when_x_zero}, $\sum_{i \in [5]} a_i > \frac{1}{3} - 2b$ from Claim~\ref{cla:gen:sum_of_ai} and summing up all the inequalities, we obtain
\[2 > 2\sqrt{\frac{1}{9} + 2b\left(\frac{1}{3}-b\right)} + 3\sqrt{\frac{1}{9} + 2b^2} + \frac{1}{3} - 2b,\]
which implies that $b > 0.39$. On the other hand, since $a_1=a_2=0$ Claim~\ref{cla:gen:numbij} implies that there are at least three non-zero values among $b_{ij}$, so $b \leq 1/3$. This gives a contradiction.  
\end{proof}

The above claim allows to use the better bound in Claim~\ref{cla:gen:sum_of_ai}. Now we show that not all $a_i$ for $i \in [5]$ can be positive. For better readability we split the proof into two claims.

\begin{claim}\label{cla:5col_onlyonebij}
If $a_i > 0$ for all $i\in [5]$, then the set $\{b_{ij}: i, j \in [5], i \neq j, b_{ij} > 0\}$ has exactly one element.
\end{claim}

\begin{proof}
Let us assume the contrary. Since $a_i > 0$ for every $i \in [5]$, we can remove a unit weight from each vertex of weight $a_i$ and add the removed weights to each vertex of weight $a_i, b_{ij}, i, j \in [5], i \neq j$ and $x$ proportionally to its weight (i.e., the weights $5a_i$, $5b_{ij}$ and $5x$ are added respectively). For every~$i \in [5]$ the increment in color $i$ is $5d_i - c_i$ and from the maximality of $G$, there must be a color $\ell \in [5]$ such that $5d_\ell - c_\ell \leq 0$. Without loss of generality let $\ell = 5$, and so $c_5 > \frac{5}{9}$. 

Similarly, by removing a unit weight from each vertex of weight $a_i$ for $i \in [4]$, and adding the removed weights to each vertex proportionally to its weight, we must have a color $\ell \in [4]$ such that the increment in color $\ell$ is non-positive, without loss of generality $\ell = 4$. It implies $c_4 > \frac{4}{9}$.

As we assumed that the set $\{b_{ij}: i, j \in [5], i \neq j, b_{ij}\}$ has at least two elements, there are at least three values of $i\in [5]$ for which $b_i>0$. Thus we can apply Claim~\ref{cla:5col_bij} to obtain a lower bound for~$c_j$ for some $j \in [3]$. For the remaining yet unbounded values of $c_i$ we can apply Claim \ref{cla:gen:basic_bound}. Finally, applying Claims \ref{cla:gen:sum_of_ai} and \ref{cla:gen:basic_bound} we get that $\sum_{i \in [5]} a_i > \sqrt{\frac{1}{9}+x^2}$. By summing up all of the above bounds and using the fact in their sum term $x$ is counted $5$ times and any other term out of $a_i$, $b_{ij}$, $i, j \in [5], i \neq j$ is counted twice, we obtain
\[2 + 3x > \frac{5}{9} + \frac{4}{9} + \frac{1}{3}\sqrt{1-6x+27x^2 + 6x\sqrt{1 - 6 x + 18 x^2}} + 3\sqrt{\frac{1}{9}+x^2}.\]
This inequality has no solutions, which ends the proof.
\end{proof}

\begin{claim}
There exists $i \in [5]$ such that $a_i = 0$.
\end{claim}

\begin{proof}
Assuming the contrary, i.e., that $a_i > 0$ for every $i \in [5]$ and using Claim~\ref{cla:5col_onlyonebij} we have, without loss of generality, that $b_{45}$ is the only non-zero value among $b_{ij}$, so $b_i = 0$ for $i \in [3]$. 
By shifting weights between $a_4$ and $a_5$, as well as between $a_1$, $a_2$ and $a_3$, we may assume that $c_1=c_2=c_3$ and $c_4=c_5$.

We bound each $c_i$ and $\sum_{i \in [5]} a_i$ similarly as in the proof of Claim~\ref{cla:5col_onlyonebij}. 
Removing a unit weight from each vertex of weight $a_i$ and adding the weights taken to every vertex proportionally to its weight, we obtain that there exists $i \in [5]$ such that $c_i > \frac{5}{9}$. 
Since $c_1=c_2=c_3$ and $c_4=c_5$, we get that the bound $c_i > \frac{5}{9}$ must hold for at least two values of $i$. 
For the remaining three values of $c_i$ and for $\sum_{i \in [5]} a_i$ we use Claims \ref{cla:gen:basic_bound} and \ref{cla:gen:sum_of_ai}. By summing up all the bounds, we obtain
\[2 + 3x > \frac{5}{9} + \frac{5}{9} + 4\sqrt{\frac{1}{9}+x^2},\]
which implies $x > 0.31$.

On the other hand, using Claim~\ref{cla:5col_bij} to bound $c_4$ and $c_5$, Claim \ref{cla:gen:basic_bound} to bound $c_1$, $c_2$, $c_3$ and bounding $\sum_{i \in [5]} a_i$ as previously, we obtain
\[2 + 3x > \frac{2}{3}\sqrt{1-6x+27x^2 + 6x\sqrt{1 - 6 x + 18 x^2}} + 4\sqrt{\frac{1}{9}+x^2},\]
which means $x < 0.27$.

The proven two bounds on $x$ give a contradiction.
\end{proof}

Knowing that there exists $i \in [5]$ such that $a_i = 0$ we can finish the proof of Theorem \ref{thm:main_weight} for $k = 5$. 
Without loss of generality we can assume that $a_5 = 0$ and bound $b_5$ from Claim~\ref{cla:5col_bi_when_ai_empty}. For $i \in [4]$ we bound $c_i$ and $\sum_{j \in [5]} a_j$ using Claims \ref{cla:gen:basic_bound} and \ref{cla:gen:sum_of_ai}. By summing up all the bounds we obtain
\[2 + 2x > \frac{1}{3}\sqrt{- 5 + 30x - 54x^2 + (6 - 12x)\sqrt{1-6x+18x^2}} + 5\sqrt{\frac{1}{9}+x^2},\]
which implies $x < 0.27$ or $x > 0.35$. From Claim \ref{cla:gen_bound_for_x} we have $x < \frac{1}{3}$, so $x < 0.27$.

Now, there are two cases that need to be considered. First let us assume that all $a_i$ for $i \in [4]$ are non-zero. By removing a unit weight from each vertex of weight $a_i$ for $i \in [4]$ and adding the removed weights to each vertex proportionally to its weight, we obtain that there is a color (without loss of generality~$1$) such that $c_1 > \frac{4}{9}$. Now without loss of generality $c_2 \geq c_3$ and from Claim \ref{cla:gen:two_ci's} we obtain that $c_2 > \sqrt{\frac{1}{9} - \frac{1}{3}x} + x$. For $b_5$ we again use Claim \ref{cla:5col_bi_when_ai_empty}. For $c_3, c_4$ and $\sum_{i \in [5]} a_i$ we use inequalities from Claims \ref{cla:gen:basic_bound} and \ref{cla:gen:sum_of_ai}. By summing up all these inequalities we obtain
\[2 + 2x > \frac{4}{9} + \frac{1}{3}\sqrt{-5+30x -54x^2 + (6 - 12x)\sqrt{1-6x+18x^2}}+ \sqrt{\frac{1}{9}-\frac{1}{3}x} + x + 3\sqrt{\frac{1}{9}+x^2}.\]
This implies $x > 0.28$, which is a contradiction to the fact that $x < 0.27$.

The second case is that there are two values $a_i$ (without loss of generality $a_4$ and $a_5$) which are equal to 0. For $b_4$ and $b_5$ we use the bound from Claim~\ref{cla:5col_bi_when_ai_empty}, for two bigger values out of $c_1, c_2, c_3$ (without loss of generality, for $c_2$ and $c_3$) the inequality from Claim \ref{cla:gen:two_ci's} and for $c_1$ and $\sum_{i\in[5]} a_i$ the bound from Claims \ref{cla:gen:basic_bound} and \ref{cla:gen:sum_of_ai}. By summing up all these inequalities we obtain
\[2 + x > \frac{2}{3}\sqrt{-5+30x -54x^2 + (6 - 12x)\sqrt{1-6x+18x^2}}+ 2\sqrt{\frac{1}{9}-\frac{1}{3}x} + 2x + 2\sqrt{\frac{1}{9}+x^2},\]
which implies $x > 0.33$. That gives a contradiction and finishes the proof for five colors.

The proof for $k = 6$ follows from the theorem for $k=5$ since $f(6) = f(5)$, analogically to the case of $k = 4$. 
By removing the vertex of weight $a_6$, edges in color $6$, and by scaling the weights of the other vertices we obtain a clustered graph for $5$ colors with at least the same density of edges in each color. Thus, a hypothetical counterexample for 6 colors implies a counterexample for 5 colors, which does not exist. 

\section{At least seven colors}\label{sec:7+}

We start the proof for $k\geq 7$ with justifying that $x$ must be positive. By contrary suppose that $x=0$. Then $c_i > \sqrt{\frac{1}{2k - 1} + 2b^2}$ for $i \in [k]$ from Claim~\ref{cla:gen:aibi_when_x_zero} and $\sum_{i \in [k]} a_i \geq \sqrt{\frac{1}{2k-1} + 2b^2} - 2b$ from Claim~\ref{cla:gen:sum_of_ai}. Summing up all these inequalities leads to
\[2 > (k+1)\sqrt{\frac{1}{2k-1} + 2b^2} - 2b.\]
The function $[7; +\infty) \ni k \longmapsto (k+1)\sqrt{\frac{1}{2k-1} + 2b^2} - 2b \in \mathbb{R}$ is increasing, so
\[2 > 8\sqrt{\frac{1}{13} + 2b^2} - 2b,\]
which is a contradiction.

In the remaining proof we consider separately the cases $k=7$, $k=8$ and $ k\geq9$.

Firstly we consider the case $k = 7$. Without loss of generality we may assume that $c_7 = \min_{i \in [7]} c_i$. From Claim \ref{cla:gen:two_ci's} we have $c_i = \max\{c_7, c_i\} > \sqrt{\frac{1}{13} - \left(1 - \frac{3}{\sqrt{13}}\right)x} + x$ for $i \in [6]$, while from Claim \ref{cla:gen:basic_bound}, $c_7 > \sqrt{\frac{1}{13} + x^2}$, and from Claim \ref{cla:gen:sum_of_ai}, $\sum_{i \in [7]} a_i > \sqrt{\frac{1}{13} + x^2}$. Summing up these inequalities we obtain
\[ 2 + 5x > 6\sqrt{\frac{1}{13} - \left(1-\frac{3}{\sqrt{13}}\right)x} + 6x + 2\sqrt{\frac{1}{13} + x^2}.\]
This implies $x > 0.38$. On the other hand from Claim \ref{cla:gen_bound_for_x} we have $x < 0.17$, which gives a contradiction.

The proof for $k = 8$ is similar but requires considering three cases depending on the number of non-zero values of $a_i$. In each case we will obtain a contradiction with $x<0.23$ from Claim~\ref{cla:gen_bound_for_x}.

If for each $i \in [8]$, $a_i > 0$, then consider removing a unit weight from each vertex of weight $a_i$ and adding the removed weights to each vertex proportionally to its weight, i.e., for every vertex $v$ of weight $w$ we add to $v$ weight $8w$. For every $i \in [8]$ the increment in color $i$ is equal to $8d_i - c_i$. Maximality of $G$ implies that there must be a color, without loss of generality it is color $8$, that has non-positive increment, which means $c_8 > \frac{8}{15}$. 
Next, consider removing a unit weight from each vertex of weight $a_i$ for $i \in [7]$ and adding the removed weights to each vertex proportionally to its weight. Similarly, it implies that for some color, without loss of generality color $7$, we have $c_7 > \frac{7}{15}$. 
Now, we may assume that $c_6 = \min_{i \in [6]} c_i$ and use the bound $c_i = \max\{c_6, c_i\} > \sqrt{\frac{1}{15} - \left( 1 - \sqrt{\frac{3}{5}} \right)x} + x$ from Claim \ref{cla:gen:two_ci's} for $i \in [5]$. Additionally, from Claims \ref{cla:gen:basic_bound} and \ref{cla:gen:sum_of_ai}, $c_6 > \sqrt{\frac{1}{15} + x^2}$ and $\sum_{i \in [8]} a_i > \sqrt{\frac{1}{15} + x^2}$. By summing up all these inequalities, we obtain
\[2 + 6x > \frac{8}{15} + \frac{7}{15} + 5\sqrt{\frac{1}{15} - \left(1 - \sqrt{\frac{3}{5}}\right)x} + 5x + 2\sqrt{\frac{1}{15} + x^2},\]
which implies $x > 0.24$ and give a contradiction with Claim~\ref{cla:gen_bound_for_x}.

If for exactly one $i \in [8]$, $a_i = 0$, without loss of generality, $a_8 = 0$, then from Claim~\ref{cla:gen:aibi_when_x_zero} we have $c_8 > \sqrt{\frac{1}{15} + 2b^2} + x \geq \sqrt{\frac{1}{15}} + x$, while by removing a unit weight from each vertex of weight $a_i$ for~$i \in [7]$ and adding the removed weights to each vertex proportionally to its weight, similarly as before, we obtain, without loss of generality, that $c_7 > \frac{7}{15}$. 
The remaining values $c_i$ for $i \in [6]$ and $\sum_{i \in [8]} a_i$ we can bound as in the previous paragraph and obtain
\[2 + 6x > \sqrt{\frac{1}{15}} + x + \frac{7}{15} +  5\sqrt{\frac{1}{15} - \left(1 - \sqrt{\frac{3}{5}}\right)x} + 5x + 2\sqrt{\frac{1}{15} + x^2}.\]
It implies $x > 0.23$ and gives a contradiction.
 
In the remaining case there are at least two distinct colors $i, j \in [8]$ such that $a_i = a_j = 0$, without loss of generality $a_7 = a_8 = 0$. It gives $c_i > \sqrt{\frac{1}{15}} + x$ for $i \in \{7, 8\}$ from Claim~\ref{cla:gen:aibi_when_x_zero}. Summing these inequalities with the previous bounds leads to
\[2 + 6x > 2\sqrt{\frac{1}{15}} + 2x + 5\sqrt{\frac{1}{15} - \left(1 - \sqrt{\frac{3}{5}}\right)x} + 5x + 2\sqrt{\frac{1}{15} + x^2},\]
which implies $x > 0.24$ and finishes the proof of $k=8$.
 
Finally, consider the case of $k \geq 9$. Without loss of generality let us assume that $c_k = \min_{i \in [k]} c_i$. From Claim \ref{cla:gen:two_ci's} we obtain $c_i > \sqrt{\frac{1}{2k-1} - \left(1 - \sqrt{\frac{2k-5}{2k-1}}\right)x} + x$ for $i \in [k-1]$. Additionally, from Claim~\ref{cla:gen:basic_bound} and Claim~\ref{cla:gen:sum_of_ai}, the bounds $c_k > \sqrt{\frac{1}{2k-1} + x^2}$ and $\sum_{i \in [k]} a_i > \sqrt{\frac{1}{2k - 1} + x^2}$ hold. Summing up all these inequalities leads to
\[2 > (k-1)\sqrt{\frac{1}{2k-1} - \left(1 - \sqrt{\frac{2k-5}{2k-1}}\right)x} + x + 2\sqrt{\frac{1}{2k - 1} + x^2}.\]
As from Claim \ref{cla:gen_bound_for_x}, $x < 1 - \sqrt{\frac{2k - 5}{2k - 1}}$, the above inequality implies
\[2 > (k-1)\sqrt{\frac{1}{2k-1} - \left(1 - \sqrt{\frac{2k-5}{2k-1}}\right)^2} + 2\sqrt{\frac{1}{2k-1}}.\]
This inequality has no solutions for $k \geq 9$, which finishes the proof of Theorem~\ref{thm:main_weight}.
 
\section{Conclusion}

Since in this paper we determine for any number of colors the minimum number of edges in each color forcing a rainbow path with $3$ edges, it is natural to ask a similar question for longer rainbow paths. Namely for positive integers $k \geq \ell \geq 4$ find the asymptotically optimal value $f(k, \ell)$ such that the following statement holds.
For every $\epsilon > 0$ there exits $n_0 \in \mathbb{N}$ such that for every $n \geq n_0$ and graphs $G_1, G_2, \ldots, G_k$ on a common set of $n$ vertices, each graph having at least $(f(k, \ell) + \epsilon)\frac{n^2}{2}$ edges, there exists a path with $\ell$ edges each coming from a different graph.
Unfortunately the method presented here cannot be easily generalized for longer paths. A not tight bound was recently presented in \cite{IKLS23}.

\end{document}